\documentclass[12pt,reqno]{amsart}
\usepackage[utf8]{inputenc}
\usepackage[T1]{fontenc}

\usepackage[]{paralist}
\usepackage{amssymb}

\usepackage{thm-restate}

\usepackage[text={15cm,22cm},centering]{geometry}

\usepackage[backend=bibtex, firstinits=true, style=alphabetic, doi=false,sorting=nyt, maxnames=10]{biblatex}

\usepackage[colorlinks=true]{hyperref}
\usepackage{cleveref}

\newtheorem{theorem}{Theorem}[section]
	
	\newtheorem{lemma}[theorem]{Lemma}
	\newtheorem{proposition}[theorem]{Proposition}

\theoremstyle{definition}
	\newtheorem{definition}[theorem]{Definition}
	\newtheorem{example}[theorem]{Example}
	
\theoremstyle{remark}
	\newtheorem{remark}[theorem]{Remark}

\newcommand{\ZZ}{\mathbb{Z}}
\newcommand{\NN}{\mathbb{N}}
\newcommand{\KK}{\mathbb{K}}

\newcommand{\set}[1]{\{#1\}}
\newcommand{\with}{\,:\,}

\newcommand{\qq}[1]{``#1''}

\DeclareMathOperator{\Tor}{Tor}
\DeclareMathOperator{\sdepth}{sdepth}
\DeclareMathOperator{\depth}{depth}
\DeclareMathOperator{\spdim}{spdim}
\DeclareMathOperator{\pdim}{pdim}
\DeclareMathOperator{\rk}{rk}

\newcommand{\Bc}{\mathcal{B}}

\newcommand{\I}{\mathrm{I}}
\newcommand{\Q}{\mathrm{Q}}

\begin{document}

\title{Betti posets and the Stanley depth}

\author{Lukas Katth\"an}

\address{Goethe-Universit\"at Frankfurt, Institut f\"ur Mathematik, 60054 Frankfurt am Main, Germany}
\email{katthaen@math.uni-frankfurt.de}

\subjclass[2010]{Primary: 05E40; Secondary: 13D02,16W50.}

\keywords{Monomial ideal; LCM-lattice; Betti poset; Stanley depth; Stanley conjecture.}

\begin{abstract}
Let $S$ be a polynomial ring and let $I \subseteq S$ be a monomial ideal.
In this short note, we propose the conjecture that the Betti poset of $I$ determines the Stanley projective dimension of $S/I$ or $I$.
Our main result is that this conjecture implies the Stanley conjecture for $I$, and it also implies that
\[ \sdepth S/I \geq \depth S/I - 1.\]
Recently, Duval et al.~found a counterexample to the Stanley conjecture,
and their counterexample satisfies $\sdepth S/I = \depth S/I - 1$.
So if our conjecture is true, then the conclusion is best possible.
\end{abstract}

\maketitle

\section{Introduction}

Let $S$ be a polynomial ring and $I \subseteq S$ a monomial ideal.
In this note we consider the \emph{Stanley depth} of $S/I$ and of $I$, which is a combinatorial invariant.
We refer the reader to \cite{intro} for a short introduction to the subject and to \cite{H} for a comprehensive survey.

The \emph{lcm-lattice} $L_I$ of a monomial ideal $I \subseteq S$ is the lattice of all least common multiples of subsets of the minimal generators of $I$.
It is known that the isomorphism type of $L_I$ determines the projective dimension of $I$, cf.~\cite{GPW}.
Further, the \emph{Betti poset} $\Bc(I) \subset \ZZ^n$ is the poset of all multidegrees in which $S/I$ has non-vanishing Betti numbers.
It is known that the Betti poset is a subposet of $L_I$ and it is determined by the latter.
Recently, Tchernev and Varisco \cite{TV}, and also Clarke and Mapes \cite{CM14b} showed that Betti poset already determines the projective dimension of $I$, in fact, it even determines the full structure of the minimal free resolution.
In \cite{lcm}, Ichim, the author and Moyano Fern\'andez showed that the Stanley projective dimension of $S/I$ and $I$ are determined by the isomorphism type of $L_I$ as well.
Here, the \emph{Stanley projective dimension} of a module $M$ can be defined as $\spdim M = \dim S - \sdepth M$. 
In the present paper, we propose the following extension of that result:
\begin{restatable*}{conjecture}{qbig}
\label{q:big}
	The Betti poset of a monomial ideal $I$ determines the Stanley projective dimension of $S/I$ and $I$.
	
	More precisely, if $I \subseteq S$ and $I' \subseteq S'$ are two monomial ideals in two polynomial rings $S$ and $S'$ such that $\Bc(I) \cong \Bc(I')$,
	then it holds that $\spdim_S S/I = \spdim_{S'} S'/I'$ and $\spdim_S I = \spdim_{S'} I'$.
\end{restatable*}
\noindent The significance of this conjecture stems from the following result:
\begin{restatable*}{theorem}{mainthm}
\label{thm:main}
	If \Cref{q:big} is true, then for any monomial ideal $I \subset S$, it holds that
	\begin{align*}
		\sdepth S/I &\geq \depth S/I - 1 \quad\text{and}\\
		\sdepth I &\geq \depth I.
	\end{align*}
\end{restatable*}
The original motivation for the research on the Stanley depth is the Stanley conjecture \cite[Conjecture 5.2]{St}, 
which asserts that $\sdepth M \geq \depth M$ for every $\ZZ^n$-graded finitely generated $S$-module $M$.
Very recently, the Stanley conjecture was disproven by Duval, Goeckner, Klivans and Martin \cite{counterexample}.
Indeed, these authors construct a monomial ideal $I$ in some polynomial ring $S$, such that 
\[ 
	\sdepth S/I = \depth S/I - 1.
\]
Thus, our \Cref{q:big} would imply the Stanley conjecture for ideals,
and it would give the best-possible bound for the Stanley depth of cyclic modules $S/I$.

We also show that \Cref{q:big} can be reduced to the following special case:
\begin{restatable*}{conjecture}{conestep}
\label{conj:onestep}
	Let $I \subset S = \KK[x_1,\dotsc,x_n]$ be a squarefree monomial ideal, let further $I' := (I\colon x_n)$ and
	assume that $\Bc(I) \cong \Bc(I')$.
	Then it holds that $\sdepth S/I = \sdepth S/I'$, or equivalently $\spdim S/I = \spdim S/I'$.
	Similarly, it holds that $\sdepth I = \sdepth I'$, or equivalently $\spdim I = \spdim I'$.
\end{restatable*}

This note is structured as follows.
In \Cref{sec:prelim} we review some background necessary for stating \Cref{q:big}.
Also, we add some remarks about this conjecture.
In the subsequent \Cref{sec:disc} we provide the proof of \Cref{thm:main} and of the equivalence of \Cref{q:big} with \Cref{conj:onestep}.
Finally, in Subsection \ref{ssec:generic} we show that a weak version of \Cref{q:big} holds for generic ideals.

\section{Statement of the conjecture}\label{sec:prelim}
Throughout the paper, let $\KK$ denote some fixed field.
By $S$ and $S'$ we denote polynomial rings over $\KK$, which we always consider with the fine grading.

\subsection{The Stanley depth}
Consider the polynomial ring $S = \KK[x_1, \dotsc, x_n]$ endowed with the fine $\ZZ^n$-grading.
Let $M$ be a finitely generated (multi-)graded $S$-module, and let $m \in M$ be a homogeneous element.
Let $Z \subset \set{x_1, \ldots , x_n}$ be a subset of the set of indeterminates of $S$.
The $\KK[Z]$-submodule $m \KK[Z]$ of $M$ is called a \emph{Stanley space} of $M$ if $m \KK[Z]$ is a free $\KK[Z]$-module.
A \emph{Stanley decomposition} of $M$ is a finite family
\[
	\mathcal{D}=(\KK[Z_i],m_i)_{i\in \mathcal{I}}
\]
in which $Z_i \subset \set{x_1, \ldots ,x_n}$ and $m_i\KK[Z_i]$ is a Stanley space of $M$ for each $i \in \mathcal{I}$ with
\[
	M \cong \bigoplus_{i \in \mathcal{I}} m_i\KK[Z_i]
\]
as a multigraded $\KK$-vector space.
This direct sum carries the structure of an $S$-module and has therefore a well-defined depth.
The \emph{Stanley depth} of $M$, $\sdepth M$, is defined to be the maximal depth of a Stanley decomposition of $M$.
Similarly, the \emph{Stanley projective dimension} $\spdim M$ of $M$ is defined as the minimal projective dimension of a Stanley decomposition of $M$.
Note that 
\[
	\spdim M + \sdepth M = n
\]
by the Auslander-Buchsbaum formula.

In the sequel, we will concentrate on modules which are either cyclic $S/I$ or ideals $I \subset S$.
In this case, Herzog, Vladoiu and Zheng \cite{HVZ} provide a convenient alternative description of the Stanley depth in terms of interval partitions.
Note that there is no known relation between $\spdim S/I$ and $\spdim I$.

\subsection{The lcm-lattice and the Betti poset}
Let $I \subset S$ be a monomial ideal.
The \emph{lcm-lattice} $L_I$ of $I$ is the lattice of all least common multiples of subsets of the minimal generators of $I$, together with a minimal element $\hat{0}$.

The following two results by Gasharov, Peeva and Welker, resp. Ichim, the author and Moyano Fern\'andez connect the lcm-lattice with projective dimension and the Stanley projective dimension.

\begin{theorem}\label{thm:origin}
	Let $I \subset S$ and $I' \subset S'$ be two monomial ideals.
	If there exists a surjective join-preserving map $L_I \rightarrow L_{I'}$, then
	\begin{align*}
		\pdim S'/I' &\leq \pdim S/I, \text{ and}  & \text{\emph{\cite{GPW}}} \\
		\spdim S'/I' &\leq \spdim S/I.  & \text{\emph{\cite{lcm}}}
	\end{align*}
	The corresponding statements hold as well for $I$ and $I'$ instead of $S/I$ and $S'/I'$.
	In particular, the isomorphism type of $L_I$ determines both the projective dimension and the Stanley projective dimension of both $S/I$ and $I$.
\end{theorem}
Here, $\pdim M$ denotes the projective dimension of $M$.
For any given finite atomistic lattice $L$, one can find a monomial ideal $I \subseteq S$ in some polynomial ring such that $L \cong L_I$, cf. \cite{P,M,lcm}.
The preceding theorem thus implies that the invariants $\pdim_\Q L := \pdim S/I, \pdim_\I L := \pdim I, \spdim_Q L := \spdim S/I$ and $\spdim_\I L := \spdim I$ do not depend on the choice of $I$.
The subscripts $\Q$ and $\I$ stand for \qq{quotient} and \qq{ideal}, respectively.

We denote by $\beta_{i,m}^S(S/I) := \dim_\KK \Tor^S_i(S/I, \KK)_m$ the multigraded Betti number of $S/I$ over $S$ in homological degree $i$ and multidegree $m$.
It is known that the Betti numbers can be computed in terms of the lcm-lattice by the following formula, cf.~\cite[Theorem 2.1]{GPW}:
\[
	\beta_{i,m}^S(S/I) =
	\begin{cases}
		\dim_\KK \tilde{H}_{i-2}(L_{< m}; \KK) &\text{ if } m \in L_I, \\
		0 &\text{ otherwise.}
	\end{cases}
\]
Here, $\tilde{H}_{i-2}(L_{< m}; \KK)$ denotes the reduced simplicial homology of the order complex of $L_{< m} \setminus \set{\hat{0}_L} = \set{n \in L\setminus \set{\hat{0}_L} \with n < m}$.
Motivated by this formula, the \emph{Betti poset} was introduced in \cite{CM14}.
\begin{definition}
	Let $L$ be a finite atomistic lattice.
	The \emph{Betti poset} of $L$ is the subset
		\[ \Bc(L) := \set{m \in L \with \tilde{H}_{i-2}(L_{< m}; \KK) \neq 0 \text{ for some }i}. \]
	Note that $\Bc(L)$ might depend on $\KK$.
	If $I \subseteq S$ is a monomial ideal, then we set $\Bc(I) := \Bc(L_I)$.
\end{definition}

It turns out that the Betti poset of a monomial ideal contains the same homological information about the ideal as the lcm-lattice:
\begin{theorem}[Theorem 5.3 of \cite{TV}, Theorem 2.1 of \cite{CM14b}]
	The Betti poset $\Bc(I)$ of a monomial ideal $I \subseteq S$ determines the structure of the minimal free resolution of $S/I$.
	In particular, it determines the Betti numbers and the projective dimension of $S/I$.
\end{theorem}
Given these results, it seems natural to ask whether the part of \Cref{thm:origin} concerning the Stanley projective dimension also extends to the Betti poset:
\qbig

\begin{remark}
\begin{asparaenum}
\item \Cref{q:big} seems a natural conjecture to us, and we have some evidence for it.
	Nevertheless, we are far from being convinced that this conjecture really holds.
	Moreover, it is possible that \Cref{q:big} holds for $S/I$ but not for $I$, or vice versa.
	In the sequel, all statements about quotients $S/I$ depend only on the part of \Cref{q:big} concerning quotients, 
	and similarly all statements about ideals $I$ depend only on the other part of \Cref{q:big}.

\item We know from \cite{IKM3} that $\pdim S/I = \spdim S/I$ for all ideals with up to five generators.
	Hence \Cref{q:big} holds for quotients of those ideals.
	Similarly, using the complete enumeration of lcm-lattices of ideals with four generators in \cite{IKM3}, we verified \Cref{q:big} for ideals with up to four generators.
\item As mentioned above, the counterexample to the Stanley conjecture by Duval et al satisfies $\sdepth S/I = \depth S/I - 1$.
	Given \Cref{thm:main} below, one could try to amplify the defect to also obtain a counterexample to \Cref{q:big}.
	One possibility would be to consider $S/I \otimes_\KK S/I$.
	However, while the depth is additive under this operation, the Stanley depth is only superadditive, i.e.,
	\[ \sdepth M \otimes_\KK N \geq \sdepth M + \sdepth N \]
	for $S$-modules $M,N$, see \cite[Proposition 2.10]{bku10} or also \cite[Theorem 3.1]{rauf2010}.
	So this does not immediately yields counterexamples to our conjecture.
\end{asparaenum}
\end{remark}

\section{Discussion of the conjecture}\label{sec:disc}

\subsection{An important consequence}
In this section we prove the following result.

\mainthm

Before we give the proof of \Cref{thm:main}, we collect some statements that we will use.
If $L$ is atomistic lattice and $a \in L$, then the \emph{rank} of $a$ is the number of atoms below it.
Further, recall that an element $a \in L$ is called \emph{meet-irreducible} if it cannot be written as a meet of two elements $b,c$ which are distinct from $a$. If $a \in L$ is meet-irreducible, then the subposet $L \setminus \set{a}$ is again a lattice.
The following is a special case of \cite[Lemma 6.4]{trees}.
\begin{lemma} 
\label{lemma:reduction}
	Let $p \in \NN$, $L$ be a finite atomistic lattice and $a \in L$ meet-irreducible.
	If $\rk a < 2p$, then it holds that $\spdim_\I L \leq \max\set{p, \spdim_\I L \setminus \set{a}}$.
\end{lemma}

	Recall that the \emph{length} $\ell = \ell(L)$ of a finite poset $L$ equals the maximal length of a strictly ascending chain 
	$ l_0 < l_1 < \dotsb < l_\ell $
	in $L$.
\begin{theorem}[Corollary 2.5, \cite{KSF}]\label{thm:length}
	For a finite atomistic lattice $L$, it holds that
	\begin{align*}
		\spdim_\Q L &\leq \ell(L) \quad\text{and}\\
		\spdim_\I L &\leq \ell(L) - 1.
	\end{align*}
\end{theorem}
\newcommand{\Lc}{\mathcal{L}}
The following proposition summarizes several results of \cite{trees} in a form which is suitable for the present purpose.
\begin{proposition}\label{prop:maximal}
	Let $I \subseteq S$ be a monomial ideal and let $p := \pdim S/I$.
	Then there exists a monomial ideal $I' \subseteq S'$ in some polynomial ring $S'$ which satisfies the following properties:
	\begin{enumerate}
		\item $\pdim S'/I' = \pdim S/I$ and (thus) $\pdim I' = \pdim I$,
		\item $\spdim S'/I' \geq \spdim S/I$ and $\spdim I' \geq \spdim I$,
		\item the length of $\Bc(I')$ is $p$, and finally
		\item $\Bc(I')$ is the face poset of an acyclic simplicial complex.
	\end{enumerate}
\end{proposition}
\begin{proof}
	Let $k$ be the number of generators of $I$. 
	Consider the set $\Lc$ of isomorphism classes of atomistic lattices with $k$ atoms.
	This is a finite poset, where the order is given by setting $L \geq L'$ if there is a surjective join-preserving map $L \to L'$.
	Consider the subposet $\Lc(p) \subseteq \Lc$ of all lattices $L$ such that $\pdim_\Q L = p$.
	Clearly, $L_I \in \Lc(p)$. As this is a finite poset, we can find a maximal element $L' \in \Lc(p)$ with $L' \geq L_I$.
	Such a lattice is called \emph{maximal} in \cite{trees}.
	Let further $I' \subset S'$ be a monomial ideal with $L_{I'} = L'$.
	By construction, it holds that $\pdim S'/I' = \pdim S/I$, and by \Cref{thm:origin} it holds that $\spdim S'/I' \geq \spdim S/I$ and $\spdim I' \geq \spdim I$.
	
	For the remaining parts of the claim we recall the description of the maximal lattices from \cite{trees}.
	By the Theorems 4.3 and 4.5 of \cite{trees}, there exists a $(p-1)$-dimensional simplicial complex $\Delta$ on $k$ vertices
	whose $(p-2)$-skeleton is complete and which is $\KK$-acyclic, such that
		\[ L' \cong \set{ F \subseteq[k] \with \Delta|_F \text{ is $\KK$-acyclic}}.\]
	Further, the Betti poset of $I'$ coincides with the face poset of $\Delta$ (cf.~\cite[Corollary 4.4]{trees}).
	So the last claim follows and for the penultimate claim we note that $\ell(\Bc(I')) = \dim \Delta + 1 = p$.
\end{proof}

\begin{proof}[Proof of \Cref{thm:main}]
	Let $p := \pdim S/I$.
	We may replace the ideal $I$ by the ideal $I'$ of \Cref{prop:maximal} without changing the validity of the claim.
	Let $L' := \Bc(I') \cup \set{\hat{1}}$, where $\hat{1}$ is a new maximal element.
	By part (4) of \Cref{prop:maximal}, $L'$ is an atomistic lattice with $\Bc(L') = \Bc(I')$.
	So we can find another monomial ideal $I'' \subset S''$ in some polynomial ring with $L_{I''} = L'$.
	By our assumption on \Cref{q:big}, it follows that $\spdim S'/I' = \spdim S''/I''$ and $\spdim I' = \spdim I''$.

	Note that the length of $L'$ equals $p+1$.
	Hence, using \Cref{thm:length} we can conclude that
	\[
	\spdim S'/I' = \spdim S''/I'' \leq \ell(L') = p+1 = \pdim S'/I' + 1.
	\]
	So the claim for $S'/I'$ is proven.
	
	It remains to show the claim for $I'$.
	If $p = 2$, then it holds that $\sdepth I' \geq \depth I'$ by \cite[Corollary 7.2]{trees} (see also \cite[Lemma 4.3]{KSF}).
	So we may assume that $p > 2$.
	We will use \Cref{lemma:reduction}.
	For this, note that $L'$ is a graded poset of rank $p+1$.
	Hence every element $a \in L'$ of rank $p$ is meet-irreducible.
	Further, $p>2$ implies that $\rk a = p < 2(p-1)$.
	
	So we conclude with \Cref{lemma:reduction} 
	that $\spdim_\mathrm{I} L' \leq \max(p-1, \spdim_\mathrm{I} L'\setminus\set{a})$.
	Iterating this procedure, we can remove all elements of rank $p$ from $L'$ and obtain a lattice $\tilde{L}$ of length $p$.
	In conclusion, we have that 
	\[
		\spdim I' = \spdim_\mathrm{I} L' \leq \max(p-1, \spdim_\mathrm{I} \tilde{L}) \leq \max(p-1, \ell(\tilde{L}) - 1) = p - 1 = \pdim I'. \qedhere
	\]
\end{proof}

\subsection{An explicit version of the conjecture}
We believe that the following more explicit formulation might be helpful in proving \Cref{q:big}.
\conestep

\begin{proposition}
	\Cref{conj:onestep} is equivalent to \Cref{q:big}.
\end{proposition}
The construction of $M(B)$ in the following proof is taken from Section 6 of \cite{TV}.
\begin{proof}
	\Cref{conj:onestep} is clearly a special case of \Cref{q:big}, so we only need to prove one implication.
	
	Let $I \subseteq S$ be a monomial ideal and set $L := L_I$. 
	Denote by $A \subset L$ the set of atoms of $L$ and let further $\Sigma(A)$ be the boolean algebra on $A$.
	There is an injective meet-preserving map $j: L \to \Sigma(A)$, which maps an element to the set of atoms below it.
	We consider $L$ as a subset of $\Sigma(A)$ via $j$.
	
	Let $B := j(\Bc(I))$ and let further $M(B) \subset \Sigma(A)$ be the set of all meets in $\Sigma(A)$ of subsets of $B$.
	Here, we consider the maximal element of $\Sigma(A)$ as the meet of the empty set.
	Then $M(B)$ is an atomistic lattice (\cite[Lemma 6.1]{TV}), and the inclusion $M(B) \subseteq L$ 
	preserves the meet. 
	Further, it holds that $\Bc(M(B)) = B$ by \cite[Proposition 6.5]{TV}.	
	
	We order the elements $a_1, \dotsc, a_r$ of $L \setminus M(B)$ be decreasing rank and set $L_i := L \setminus \set{a_1, \dots, a_i}$.
	This way, we obtain an increasing chain 
	\[ M(B) = L_r \subsetneq L_{r-1} \subsetneq \dotsb \subsetneq L_0 = L\]
	of lattices, where all the inclusions are meet-preserving (cf. \cite[Lemma 3.7]{trees}).
	It is easy to see by induction on $i$ that $\Bc(L_i) = \Bc(L)$ for all $i$.
	Indeed, this is obvious for $i = 0$.
	For $i > 0$, it holds that $a_{i+1} \notin \Bc(L) = \Bc(L_i)$, and hence the arguments given in Lemma 3.8 and Remark 3.9 of \cite{trees} show that $\Bc(L_{i+1}) = \Bc(L_i)$.
	
	Fix an $i$ such that $0 \leq i \leq r$. It follows from \cite[Theorem 3.4]{lcm}
	that there exists a squarefree monomial ideal $J \subset S''$ in some polynomial ring $S''$ and a variable $x \in S''$, such that $L_i \cong L_J$ and $L_{i+1} \cong L_{(J:x)}$.	
	Thus, \Cref{conj:onestep} implies that $\spdim_\Q L_i = \spdim_\Q L_{i+1}$ and $\spdim_\I L_i = \spdim_\I L_{i+1}$. As this holds for all $i$, we arrive at the conclusion that
	\[ \spdim S/I = \spdim_\Q L = \spdim_\Q M(B) \]
	and 
	\[ \spdim I = \spdim_\I L = \spdim_\I M(B). \]
	Now let $I' \subset S'$ be a second monomial ideal with $\Bc(I) \cong \Bc(I')$. This clearly implies that $M(\Bc(I)) \cong M(\Bc(I'))$, and hence that
		\[ \spdim S/I  = \spdim_\Q M(\Bc(I)) = \spdim_\Q M(\Bc(I')) = \spdim S'/I' \]
	and similar for $\spdim I$.
\end{proof}

\begin{remark}
\begin{asparaenum}
\item
	Note that the inequalities $\sdepth_S S/I \leq \sdepth_S S/I'$ and also $\sdepth I \leq \sdepth I'$ are clear,
	because every Stanley decomposition of $S/I$ and $I$ restricts to a Stanley decomposition of $S/I'$ and $I'$, respectively. 
	So the difficulty is to extend a Stanley decomposition of $S/I'$ or $I'$ to a Stanley decomposition of $S/I$ or $I$.

\item
	For \Cref{thm:main}, it would be enough to prove \Cref{conj:onestep} (and thus \Cref{q:big}) for those ideals which actually appear in the proof of \Cref{thm:main}.
	In particular, one may assume that the minimal free resolution of $I$ is supported on the Scarf complex, and that the latter is a stoss complex in the sense of \cite{trees}, i.e. an acyclic $(p-1)$-dimensional simplicial complex with a complete $(p-2)$-skeleton, where $p = \pdim S/I$.
	Note that such a resolution is a truncation of the Taylor resolution.

\item
	The assumption that $I$ is squarefree is inessential, as it does not affect the lcm-lattices.
	One may further assume in \Cref{conj:onestep} that all generators of $I$ have the same degree, cf. \cite[Proposition 5.12]{lcm}.
\end{asparaenum}
\end{remark}

It seems desirable to understand the implications of the condition $\Bc(I) \cong \Bc(I')$ in \Cref{conj:onestep}.
One possible approach is to consider the Hilbert series of $S/I$. Recall that it is given by
\[ \mathbf{H}(S/I; t_1, \dotsc, t_n) = \frac{1}{(1-t_1)\dotsm(1-t_m)}\sum_{m \in \ZZ^n}t^m \sum_{i\geq 0} (-1)^i \beta_{i,m}(S/I), \]
where we write $t^m = t_1^{m_1}\dotsm t_n^{m_n}$.
Now the condition $\Bc(I) \cong \Bc(I')$ implies that the Hilbert series of $S/I'$ has the same \qq{shape}, in the sense that no further cancellation of terms occurs.
As Stanley decompositions can be seen a decompositions of the Hilbert series, this might imply that the Stanley decompositions are also similar.
However, the following example shows that it is not enough to consider this \qq{shape} of the Hilbert series.
\begin{example}
Let $S = \KK[a,b,c,x,y]$ and consider the following two ideals
\begin{align*}
	I_1 &:= (a^2x^2,b^2x^2,c^2x^2,a^2b^2c^2,abcxy)\\
	I_2 &:= (a^2x^2,b^2x^2,c^2x^2,a^2b^2c^2,abcx)
\end{align*}
in $S$.
One can easily compute that $\pdim S/I_1 = 4$ and that $\pdim S/I_2 = 3$.
Further, both ideals have only five generators, so their Stanley projective dimensions coincide with the respective projective dimensions \cite{IKM3}.
In particular, their Betti posets are nonisomorphic and their Stanley projective dimensions differ.
Further, using the algorithm of \cite{IZ}, we computed that $\spdim I_1 = \spdim I_2$.

On the other hand, their Hilbert series have the same \qq{shape}.
The reason is that for all elements $m \in \Bc(I_1) \setminus \Bc(I_2)$ it holds that $\sum_{i \geq 0}(-1)^i \beta_{i,m}(S/I_1) = 0$.
\end{example}

\subsection{Generic ideals}\label{ssec:generic}
In this section we show that a weak version of \Cref{q:big} holds for generic ideals in the sense of Miller, Sturmfels and Yanagawa \cite{MSY}.
\begin{proposition}
	Let $I \subset S$ be a generic monomial ideal. Let further $I' \subseteq S'$ be a further monomial ideal,
	such that there is a surjective join-preserving map $L_I \to L_{I'}$ and assume that $\Bc(I) \cong \Bc(I')$.
	
	Then it holds that $\spdim_S S/I = \spdim_{S'} S'/I'$.
\end{proposition}
The addition assumption that there is a map  $L_I \to L_{I'}$ is not a severe restriction, because in our proof of \Cref{thm:main} we only consider this situation.
\begin{proof}
	Let $p = \pdim_S S/I$.
	Recall that the Scarf complex of $I$ is the subset $\Delta(I) \subseteq L_I$ of those elements which can be written as a join of atoms in a unique way.
	In general, the Betti poset contains the Scarf complex, and in the generic situation these two coincide.
	So every element of the Betti poset of $I$ is a join of atoms in a unique way.
	But then this also holds for $I'$, and hence $\Bc(I') = \Delta(I')$ as well.
	In particular, there exists an element $a \in \Delta(I')$ of rank $p = \pdim_S S/I = \pdim_{S'} S'/I'$.
	We can construct a surjective join-preserving map from $L_{I'}$ to a boolean algebra on $p$ atoms by sending every element $b \in L_{I'}$ to the set of atoms below $b \wedge a$.
	The boolean algebra on $p$ atoms can be considered as lcm-lattice of an ideal $J$ generated by $p$ variables, and the latter has Stanley projective dimension $p$.
	Hence we conclude with \Cref{thm:origin} that $\spdim_{S'} S'/I' \geq p$.
	On the other hand, the Stanley conjecture holds for $S/I$ (as $I$ is generic), so $\spdim_S S/I \leq p$.
	Hence, using again \Cref{thm:origin} it follows that
	\[ p\leq \spdim_{S'} S'/I' \leq \spdim_S S/I \leq p, \]
	so the claim follows.
\end{proof}

\printbibliography

\end{document}